\documentclass[12pt, reqno]{amsart}

\usepackage{amsmath}
\usepackage{amssymb} 
\usepackage{array}
\usepackage{enumitem}
\usepackage{color}
\definecolor{battleshipgrey}{rgb}{0.52, 0.52, 0.51} 
\usepackage{hyperref}
\hypersetup{
    bookmarks=true,         
    unicode=false,          
    pdftoolbar=true,        
    pdfmenubar=true,        
    pdffitwindow=false,     
    pdfstartview={FitH},    
    pdftitle={My title},    
    pdfauthor={Author},     
    pdfsubject={Subject},   
    pdfcreator={Creator},   
    pdfproducer={Producer}, 
    pdfkeywords={Jordan algebra} {cyclic order} {partially ordered ring}, 
    pdfnewwindow=true,      
    colorlinks=true,       
    linkcolor=battleshipgrey,          
    citecolor=battleshipgrey,        
    filecolor=battleshipgrey,      
    urlcolor=battleshipgrey          
}
\usepackage{verbatim} 

\usepackage{pstricks-add}

\theoremstyle{plain}
\newtheorem{theorem}{Theorem}[section]
\newtheorem{lemma}[theorem]{Lemma}
\newtheorem{proposition}[theorem]{Proposition}

\newtheorem{definition}[theorem]{Definition}
\theoremstyle{remark}

\newtheorem{example}{Example}[section]

\newtheorem*{acknowledgment}{Acknowledgment}

\numberwithin{equation}{section}

\newcommand{\bA}{\mathbb{A}}

\newcommand{\K}{\mathbb{K}}

\newcommand{\R}{\mathbb{R}}

\newcommand{\Q}{\mathbb{Q}}
\newcommand{\Z}{\mathbb{Z}}
\newcommand{\N}{\mathbb{N}}
\newcommand{\PP}{\mathbb{P}}

\newcommand{\Gl}{\mathrm{Gl}}

\newcommand{\Sym}{\mathrm{Sym}}

\newcommand{\Aut}{\mathrm{Aut}}

\newcommand{\End}{\mathrm{End}}

\newcommand{\Co}{\mathrm{Co}}

\newcommand{\id}{\mathrm{id}}

\newcommand{\Istr}{\mathrm{Istr}}
\newcommand{\Str}{\mathrm{Str}}

\newcommand{\Cau}{\mathrm{Cau}}

\newcommand{\eps}{\varepsilon}

\newcommand{\ssk}{\smallskip}
\newcommand{\nin}{\noindent}

\begin{document}

\title{Cyclic orders defined by ordered Jordan algebras}


%

\author{Wolfgang Bertram}
\address{Institut \'{E}lie Cartan de Lorraine \\
Universit\'{e} de Lorraine at Nancy, CNRS, INRIA \\
B.P. 70239 \\
F-54506 Vand\oe{}uvre-l\`{e}s-Nancy Cedex, France}

\email{\url{wolfgang.bertram@univ-lorraine.fr}}

\begin{abstract}
We define a general notion of  {\em partially ordered Jordan algebra} (over a partially ordered ring), and
we show that the Jordan geometry associated to such a Jordan algebra admits a natural  invariant
{\em partial cyclic order}, whose intervals are modelled on the {\em symmetric cone} of the Jordan algebra.
We define and describe, by affine images of intervals, the {\em interval topology} on the Jordan geometry,
and we outline a research program  aiming at generalizing main features of the theory of classical symmetric cones and
bounded symmetric domains.
\end{abstract}



\maketitle

\section*{Introduction}

Some algebraic structures fit well with {\em partial orders}, and others less.
Thanks to their algebraic origin by ``squaring operations'',
{\em Jordan algebras} have a very privileged
 relation with partial orders, and this interplay has been studied for a long time.
However, to my best knowledge, so far no general notion of {\em partially ordered Jordan algebra} has appeared in the
literature: only special cases, the (finite dimensional) {\em Euclidean Jordan algebras} (cf.\ \cite{FK94}) and
their analogs in the Banach-setting (cf.\ \cite{Up}) have been thoroughly studied.
A first aim of this paper is to give general definitions, following the ``partial order-philosophy'' 
(po-phi), of the following items:

\begin{itemize}
\item
 por: partially ordered ring (Def.\ \ref{def:por}),
\item
pom: partially ordered module (over a por; Def.\ \ref{def:pom}), with special case the well-known
povs (partially ordered vector spaces),
\item
poJa: partially ordered Jordan algebra (a  pom over a por whose {\em quadratic operators} preserve the
partial order; Def.\ \ref{def:poJa}),
\item
pco: 
partial cyclic order ( cf.\ Def.\ \ref{def:pco}).
\end{itemize}
The last item leads to the second topic of this work:
cyclic orders. As Coxeter puts it (\cite{Co47}, p. 31):
{\it 
The intuitive idea of the two opposite directions along a line, or a round circle, is so familiar that we are apt to
overlook the niceties of its theoretical basis.} Indeed,
``geometric'' spaces often look somehow like a circle,  are compact, or compact-like, and
cannot be reasonably ordered in the usual sense.
But, according to ideas going back to geometers of the 19-th century (cf.\ \cite{Co47, Co49}), 
linear orders can successfully be replaced by {\em cyclic orders}. The archetypical example is the circle,
$M=S^1$. Let us say that a triple $(a,b,c) \in M^3$ is {\em cyclic}, and write
$(a,b,c) \in R$, if 
``$(a,b,c)$ occur  in this  order when running counter-clockwise around the circle'', as in Figure \ref{f:1},
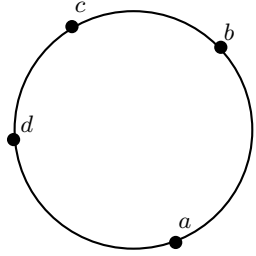
\begin{figure}[h]
\caption{The circle, with cyclically ordered $(a,b,c)$, $(b,d,a)$...}\label{f:1}
\newrgbcolor{xdxdff}{0.490196078431 0.490196078431 1.}
\psset{xunit=0.5cm,yunit=0.5cm,algebraic=true,dimen=middle,dotstyle=o,dotsize=3pt 0,linewidth=0.8pt,arrowsize=3pt 2,arrowinset=0.25}
\begin{pspicture*}(-4.3,-2.74)(14.44,4.1)
\pscircle(1.32,0.62){1.58}
\begin{scriptsize}
\psdots[dotsize=5pt 0,dotstyle=*,linecolor=black](3.64,2.82)
\rput[bl](3.72,3.02){{$b$}}
\psdots[dotsize=5pt 0,dotstyle=*,linecolor=black](-0.313239714564,3.36862293427)
\rput[bl](-0.24,3.76){{$c$}}
\psdots[dotsize=5pt 0,dotstyle=*](-1.86674736438,0.361076776644)
\rput[bl](-1.68,0.56){{$d$}}
\psdots[dotsize=5pt 0,dotstyle=*](2.44482993764,-2.37285108407)
\rput[bl](2.52,-2.00){{$a$}}
\end{scriptsize}
\end{pspicture*}
\end{figure}
where, e.g., $(a,b,c) \in R$, $(b,d,a) \in R$. 
There is a simple axiomatic definition of a general 
\href{https://en.wikipedia.org/wiki/Cyclic_order}{\em (partial) cyclic order} \footnote{ hyperlinks in grey in the 
electronic version; the wikipedia page on cyclic orders contains many further references} on a set $M$ (Def.\ \ref{def:pco}),
and the topic of general cyclic orders has been studied by many authors.

\ssk
Now, the main point of the present work is to establish a link between the preceding items:
on the one hand, Jordan algebras $V$ correspond to {\em Jordan geometries} $X = X(V)$ (I recall their definition in Section
\ref{sec:Jordangeo}); in the classical
case (Euclidean Jordan algebras), the underlying space $X$ is compact (typically, Grassmannian or Lagrangian
varieties); in the more general cases
(Banach, or general poJa) it won't be compact, but still  look ``compact-like'' -- a sort of non-compact analog of a circle.
Our main result fits  well with this impression: we show (Theorem \ref{th:main}):
{\em
Every poJa $V$ gives rise to a partial cyclic order $R$ on its Jordan geometry $X(V)$, which is natural in the
sense that it is invariant under the ``inner conformal'' or ``Kantor-Koecher-Tits group'' of $V$, and that
all intervals defined by $R$ are modelled on the symmetric cone $\Omega$ of the poJa $V$.}
We also give a description of  intervals $]a,b[$ by their {\em
affine images} (intersection with affine part $V\subset X$;
Theorem \ref{th:affineimage}): 
when $b=\infty$, 
the image  is ``parabolic'', a translation of the symmetric cone $\Omega$;
when $a,b \in V$ and $a<b$ in $V$, then the image is ``elliptic'': complete, of the form
$(a+\Omega) \cap (b-\Omega)$, see Figure \ref{f:2}.
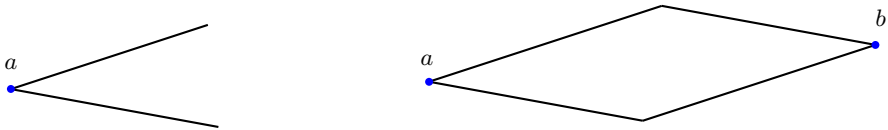
\begin{figure}[h]
\caption{Images of an interval:
parabolic  (left), elliptic  (right).}\label{f:2}
\newrgbcolor{zzttqq}{0.6 0.2 0.}
\psset{xunit=0.7cm,yunit=0.7cm,algebraic=true,dimen=middle,dotstyle=o,dotsize=3pt 0,linewidth=0.8pt,arrowsize=3pt 2,arrowinset=0.25}
\begin{pspicture*}(-2.3,0.5)(16.3,3.7)
\psline(-1.46,1.48)(2.28,2.7)
\psline(-1.46,1.48)(2.48,0.76)
\psline(6.48,1.62)(10.9002176116,3.06188916742)
\psline(10.9002176116,3.06188916742)(14.96,2.32)
\psline(14.96,2.32)(10.5397823884,0.878110832578)
\psline(6.48,1.62)(10.5397823884,0.878110832578)
\begin{scriptsize}
\psdots[dotstyle=*,linecolor=blue](-1.46,1.48)
\rput[bl](-1.58,1.86){{$a$}}
\psdots[dotstyle=*,linecolor=blue](6.48,1.62)
\rput[bl](6.32,1.94){{$a$}}
\psdots[dotstyle=*,linecolor=blue](14.96,2.32)
\rput[bl](14.96,2.68){{$b$}}
\end{scriptsize}
\end{pspicture*}
\end{figure}

\nin
In all other cases for $(a,b)$, the affine image of $]a,b[$ is ``hyperbolic'' -- 
hard to represent in a two-dimensional drawing; see Figure \ref{f:3} in the very simple example of a
torus (Example \ref{ex:torus}) for a special case. 
The proofs of these two theorems
 are straightforward and do not need any deep analytic or order-theoretic tools --
the main ideas can already be found in 
Section XI of \cite{Be00} (if I had known about the notion of cyclic order at that time, I certainly would have 
formulated Theorem XI.3.3 loc.cit.\ in these terms).
The intervals and their affine images are used to define and study the {\em interval}, or {\em order topology}
on $X$, which simply is the topology generated by the intervals (Def.\ \ref{def:intervaltopology}).
In the classical cases, it coincides with the usual topology (Theorem \ref{th:topology}).
In the final Section \ref{sec:further}, I outline a list of open problems, which in my opinion represents a rather
promising research program aiming to generalize  the classical theory of
symmetric cones and bounded symmetric domains:

\begin{enumerate}
\item
{\sl generalized tube domains and bounded symmetric domains,}
\item
{\sl
compact dual: symmetric $R$-spaces, Borel imbedding,}  
\item
{\sl  define and study the boundary of symmetric cones and intervals;}
\item
{\sl structure theory of poJas's}, including notions of
{\sl traces, and states,}
\item
{\sl relation with invariats (cross-ratio, Maslov index)},
\item
{\sl duality}; and {\sl what makes a cone ``symmetric''?}
\end{enumerate}

\nin
As to the last item,
we have used above the term {\em symmetric cone} by generalizing the classical one (\cite{FK94}); indeed, 
the intervals $]a,b[$ of the cyclic order on $X$
 are symmetric about any of their points (for any $y \in ]a,b[$ there is an order-reversing
bijection of order two fixing $y$), hence the term ``symmetric cone'' seems well deserved.
However, it does not imply existence and isomorphism with some ``dual cone in a dual vector space''. 
This degree of generality may seem excessively wide for many readers, but
I think the domains of Jordan theory and order theory are today ripe enough to be treated in full generality.
The principal merit of this degree of generality may be to clarify the interplay between geometry, algebra,
and analysis in this realm, and, hopefully, to lead to a better understanding of all three of them.

\begin{acknowledgment}
This work has been triggered by discussions during the workshop 
``\href{http://www.lorentzcenter.nl/lc/web/2017/897/info.php3?wsid=897&venue=Snellius}{Order Structures, Jordan Algebras and Geometry}'' held in Lorentz Center Leiden in may-june 2017, and I would like to thank the organizers and the staff of the 
Lorentz Center for
making possible this pleasant and fruitful workshop.
\end{acknowledgment}

\section{Ordered rings and  algebras}

\subsection{Partial orders}
A {\em partial order on a set $M$} is as usual defined  to be a binary relation $<$  that is
asymmetric and transitive.
If we denote (the graph of) this relation by $L = \{ (a,b) \in M^2 \mid a<b\}$, then
asymmetry means $L \cap L^{-1}=\emptyset$ and transitivity $L \circ L \subset L$, where
$\circ$ is relational composition, and $L^{-1}$ is the reverse of a relation $L$.
The relation
$ L^{eq} := L \cup \Delta_M
$,
where
$\Delta_M=\{ (a,a)\mid a \in M\}$ is the (graph of) the identity relation, is denoted as usual
by $\leq$.
For $(a,b) \in M^2$, the  {\em (open, resp. closed) interval} between $a$ and $b$ is denoted 
in the ``French way''  by
\begin{align*}
]a,b[ & := \{ x \in M \mid a <x < b \} = \{ x \in M \mid (a,x) \in L, (x,b) \in L \} ,
\\
[a,b] & := ]a,b[ \cup \{ a, b \} = \{ x \in M \mid a \leq x \leq b \}.
\end{align*}
We speak of a {\em total}
linear order, if, moreover, 
$$
M \times M = \Delta_M \cup L \cup L^{-1}. 
$$
In this work, by ``order'' or ``ordered'' we shall always mean ``partial order'', resp.\ ``partially ordered'',
 and we shall always work with $<$ as basic relation, rather than with $\leq$.


\subsection{Ordered and square ordered rings} 
The first part of the following definition is  \href{https://en.wikipedia.org/wiki/Partially_ordered_ring}{standard}:

\begin{definition}\label{def:por}
A {\em partially  ordered ring (por)} is a ring $(\bA,+,\cdot)$ together with a partial order $<$ on $\bA$ such that:
\begin{enumerate}
\item
$\forall a,b,c \in \bA$: $a < b \Rightarrow a + c <  b +c$,
\item
$\forall a,b,c  \in \bA$: $(0 < b$ and $a < c)  \Rightarrow  (ba < bc$ and $ab<cb)$.
\end{enumerate}

\nin
If the order is total, then the ring will be called
{\em totally ordered}.
If the ring has a unit element $1$, then we will always assume that $0 < 1$.
We say that $\bA$ is {\em square-ordered} if, moreover, invertible squares are positive in the following sense:
\begin{enumerate}
\item[(3)]
$\forall a \in \bA^\times$ (invertible elements):
$0<a^2$,
\end{enumerate}
and we say that $\bA$ is an {\em inverse por} if all positive elements are invertible:
\begin{enumerate}
\item[(4)]
$\forall a >0$: $a \in \bA^\times$.
\end{enumerate}
\end{definition}

\begin{example}\label{ex:por}
$\K = \R, \Q$ or $\Z$ with their usual order are totally (square) ordered rings (but $\Z$ is not an inverse por);
the direct product of ordered rings is an ordered ring;
rings of functions with values in an ordered ring form an ordered ring;
the {\em ring of dual numbers}
$$
\K = \R[X]/(X^2) = 
\R[\eps] = \R \oplus \eps \R, \quad (x+\eps y)(x' + \eps y')= xx' + \eps (xy' + yx')
$$
is partially ordered by letting $x + \eps y > x' +\eps y'$ iff $x >x'$. 
In the same way it is seen that, if  $\bA$ is an ordered ring, 
then $\bA[\eps] = \bA \oplus \eps \bA$ is an ordered ring. 
Properties (3) and (4) behave well with respect to these constructions.
\end{example}

\begin{example}
For an ordered field, (1) and (2) implies (3), but for an ordered ring, this need not be the case: 
consider $\K=\R$ (or $\K ={\mathbb C}$)  with the ``trivial partial order'' given by
$L = \{ (x,x+n) \mid x \in \R, n \in \N \}$; in other words, $x>0$ iff $x \in\N$. 
It satisfies (1), (2), but not (3). 
\end{example}

\begin{example}\label{exa:C}
Since $-1$ cannot be a square in a square ordered ring, the ring
$$
\bA[i]:= \bA \oplus i \bA, \quad  (x+iy)(x' + i'y') = (xx' - yy') + i (xy' + yx') 
$$
carries no structure of square ordered ring. 
When $\bA = \R[\eps]$, then 
$\bA[i]= {\mathbb C}[\eps]$, and when $\bA=\Z$, then $\bA[i]$ is the ring of Gaussian integers. 
\end{example}

\subsection{Ordered modules and convex cones}
The following is the precise analog of notions of 
\href{https://en.wikipedia.org/wiki/Ordered_vector_space}{\em (partially) ordered vector spaces}:

\begin{definition}[Ordered module]\label{def:pom} 
Assume $(\K,<)$ is a commutative por.
A $\K$-module $V$ is called a {\em partially ordered module (pom)} if it carries a partial order $<$ such that
\begin{enumerate}
\item
$\forall a,b,c \in V$: $a < b \Rightarrow a + c <  b +c$,
\item
$\forall a, c  \in V$, $\forall \beta  \in \K$: $(0 < \beta$ and $a < c)  \Rightarrow  \beta a <  \beta c$.
\end{enumerate}
The set $\Omega:= \{ a \in V \mid 0 < a \}$ is called the {\em positive cone}.
\end{definition}

\begin{proposition}
Let $V$ be a partially ordered module, and $\Omega$ as above. Then
\begin{enumerate}
\item
$\forall a,b \in \Omega : a + b \in \Omega$
\item
$\forall a \in \Omega$, $\forall \beta \in \K$, $\beta >0 \Rightarrow \beta a \in \Omega$
\item
$\Omega \cap (- \Omega) = \emptyset$.
\end{enumerate}
Conversely, given a subset $C $ of $ V$ having these properties,
$a < b$ iff $b-a \in C$ defines on $V$ the structure of   an ordered module.
The order is total iff $V = C \cup \{ 0 \} \cup (-C)$.
\end{proposition}

\begin{proof}
All standard arguments from the real case go through. Recall that (1) is related both to transitivity and to translation invariance
of the partial order. 
\end{proof}

\begin{definition}
A {\em convex cone} in a partially ordered $\K$-module $V$ is a subset $C \subset V$ satisfiying (1), (2); it is called {\em salient} if it
satisfies also (3), and {\em acute} if it contains no affine line. 
\end{definition}

\begin{definition}
A subset $S \subset V$ in an ordered $\K$-module $V$ is called {\em convex} if it satisfies one of the following equivalent
conditions:
\begin{enumerate}
\item[(c1)]
$\forall a,b \in S, \,  \forall t \in [0,1] : \quad (1-t) a + tb \in S$,
\item[(c2)]
$\forall a,b \in S, \, \forall t \in \, \, ]0,1[ \,  : \quad (1-t) a + tb \in S $.
\end{enumerate}
\end{definition}

Obviously, a convex cone is convex as a set. 
Non-convex cones are rarely considered in the literature; an exception is \cite{FG96}. 
Summing up, we have a bijection between order structures on modules over ordered rings
and salient convex cones in such modules. With the suitable definitions,
this bijection can be turned into an equivalence of categories. 
Note that translations are isomorphisms of $<$, and multiplications by positive elements defines endomorphisms of $<$.

\begin{example} In the special case $V = \K$, 
por-structures on $\K$ are in bijection with subsets $C \subset \K$ such that
$C + C \subset C$, $C \cdot C \subset C$, $C \cap (-C) = \emptyset$.
We get a square order if $a^2 \in C$ for all $a \in \K^\times$, and an inverse por if $C \subset \K^\times$.
\end{example} 

\begin{example}
Every linear form $\phi : V \to \K$ defines a partial order by letting
$a < b$ iff $\phi(a) <\phi(b)$. If $\phi \not= 0$,
the cone of $v$ is a ``wedge'', or ``tube'',  modelled on the cone of $\K$
(half-space if $\K$ is totally ordered). It is salient but not acute. 
\end{example}

\subsection{Jordan algebras}\label{ssec:Ja}
A standard reference is \cite{McC}:
let $\K$ be a commutative ring containing $\frac{1}{2}$.
A {\em linear Jordan algebra (over $\K$)} is a $\K$-module $V$ with a bilinear product
$V^2 \to V$, $(a,b) \mapsto a\bullet b$ satisfying the identities
\begin{enumerate}
\item[(J1)]
$a \bullet b = b \bullet a$,
\item[(J2)]
$a \bullet (a^2 \bullet b) = a^2 \bullet (a \bullet b)$ where $x^2 = x \bullet x$.
\end{enumerate}
We assume that $V$ contains a {\em unit element $e\not= 0$}. 
In a linear Jordan algebra, one defines the {\em quadratic operator $Q_a$} by
\begin{equation}
Q_a(x) = (2 L_a^2 - L_{a^2})(x) = 2 (a \bullet (a \bullet x)) -  a^2 \bullet x
\end{equation}
and its linearization 
\begin{equation}
D_{a,x}(b) = (Q_{a+b} - Q_a - Q_b)(x) = 2 (a(bx)-(ab)x+x(ab)) .
\end{equation}
Then the following holds (see, e.g., \cite{McC}):
\begin{enumerate}
\item[(U)] (unit) $Q_e = \id_V$,
\item[(FF)] (fundamental formula) 
$Q_{Q_a(b)} = Q_a Q_b Q_a$,
\item[(CF)] (commutation formula)
$Q_a D_{b,x} = D_{x,b} Q_a$
\end{enumerate}
By definition, a {\em quadratic Jordan algebra} is a $\K$-module $V$ together with a quadratic map
$Q:V \to \End_\K(V)$, $a \mapsto Q_a$ satisfying, in all scalar extensions of $\K$, (U), (FF) and (CF). 
This definition even makes sense when $2$ is not invertible in $\K$.
The squaring operation is recovered by
$a^2 = Q_a(e)$. 

\begin{definition}
An element $a$ of a (quadratic) Jordan algebra is called {\em invertible} if
$Q_a:V \to V$ is bijective, and its {\em inverse} is then defined by
$$
a^{-1} := Q_a^{-1}(a) .
$$
The set of invertible elements in $V$ is denoted by $V^\times$.
\end{definition}

The set $V^\times$ is stable under the binary map
$(a,b) \mapsto s_x(y) = Q_x (y^{-1}) = Q_x Q_y^{-1} y$, which satisfies the algebraic properties of a 
{\em symmetric space} from \cite{Lo69}:

\begin{enumerate}
\item
$s_x^2 =\id$
\item
$s_x(x)=x$
\item
$s_x s_y s_x = s_{s_x(y)}$
\item
the differential of $s_x$ at $x$
is $-\id_V$.
\end{enumerate}
 (The differential can be defined in a purely algebraic way by scalar extension via dual numbers.)
In finite dimension over a field, $V^\times$ is Zariski-dense in $V$, and then $V^\times$ is a
{\em quadratic prehomogeneous symmetric space} in the sense of \cite{Be00}, Chapter II. 
The term ```quadratic'' means that  the quadratic operator coincides with the quadratic representation of the
symmetric space as defined in \cite{Lo69}.
Indeed, $s_e$ is the Jordan inversion map, and
$ s_x s_e = Q_x$. 
 
\begin{example}
Any associative algebra $\bA$ becomes a  quadratic Jordan algebra with
$Q_a(b) = aba$. The Jordan powers $a^k$ agree with the usual ones. The symmetric space structure comes from
the group structure of $\bA^\times$ via $s_xy = xy^{-1} x$.
\end{example}

\subsection{Ordered Jordan algebras and their symmetric cones}
To define a notion of ordered Jordan algebra, 
it would be misleading to simply copy the definition of a partially ordered ring:
even in a Euclidean Jordan algebra, it is in general not true that $a>0$ and $b>0$ implies
$a\bullet b >0$. 
\footnote{ Here is a counterexample:
let $V = \Sym(2,\R)$ with its cone of positive definite
 symmetric matrices. The following
$A,B$ are positive definite, but $AB+BA$ is not:
$$
A = \begin{pmatrix} 2 & 1\\1 & 1 \end{pmatrix}, \quad
B = \begin{pmatrix} 1 & 0\\ 0 & 9 \end{pmatrix}, \quad 
AB+BA = \begin{pmatrix} 4 & 10 \\ 10  &  18 \end{pmatrix} \, .
$$
}

\begin{definition}\label{def:poJa}
A {\em partially  odered Jordan algebra (poJa)}  is a unital Jordan algebra $V$ which is a partially ordered module 
over a commutative por $\K$, such that:
\begin{enumerate}
\item[{\rm (OJ0)}] $e>0$ (where $e$ is the unit element),
\item[{\rm (OJ1)}] $\forall a>0$: $a \in V^\times$,
\item[{\rm (OJ2)}]
$\forall a \in V, \forall b \in V^\times$: $a>0 \Rightarrow Q_b(a)>0$.
\end{enumerate}
We call {\em symmetric cone of $V$} the  positive cone 
$$
\Omega =  \{ a \in V \mid \, a>0  \} . 
$$
\end{definition}

\begin{lemma}
The symmetric cone $\Omega$ of a poJa is a sub-symmetric space of $V^\times$: it is stable under the
binary map $(x,y) \mapsto s_x(y) = Q_x (y^{-1})$ (one may call it a {\em convex prehomogeneous symmetric space}). 
It contains all squares of invertible elements.
\end{lemma}

\begin{proof}
By (OJ1), $\Omega \subset V^\times$, and by (OJ2), 
if  $y>0$, then $y^{-1} = (Q_y^{-1})(y)=Q_{y^{-1}}(y) >0$, so
$s_x(y)>0$ for all $x \in V^\times$.
Since $e>0$, we have $x^2 = Q_x(e) >0$ for all invertible $x$, by (OJ2).
\end{proof}

\begin{example}
Let $V=\K=\Q$. The smallest possible symmetric cone is
$\Omega_0 = \{ \sum_{i=1}^n  x_i^2 \mid n\in \N,  \forall i: x_i \in \Q^\times  \}$
which equals the positive cone for the usual (total) order. Note that this cone is strictly bigger than the set of
invertible squares.  
\end{example}

\begin{example}
Any real Euclidean Jordan algebra $V$ is a poJa, and $\Omega$ coincides with its associated symmetric cone 
as defined in \cite{FK94}
(and hence our terminology should not lead to conflicts). 
Indeed, in this case $\Omega = \{ x^2 \mid x \in V^\times\}$.
Likewise, for a Jordan-Banach algebra, $\Omega$ coincides with its usual symmetric cone.
\end{example}

\begin{example}[Positive definite quadratic forms]
All classical Euclidean Jordan algebras can also be defined over the field $\Q$, and over any ordered
ring extension $\K$ of $\Q$, such as dual numbers.
For instance, the classical Jordan algebra $\Sym(n,\R)$ of symmetric $n\times n$-matrices gives rise to 	a $\K$-poJa
 $\Sym(n,\K)$.

\ssk
More conceptually: if $\K$ is an ordered ring, a {\em quadratic form on $\K^n$} is a map
$q:\K^n \to \K$, homogeneous of degree 2 and such that
$d(x,y):=q(x+y)-q(x)-q(y)$ is bilinear; we say that $q$ is {\em positive definite} if
$q(v) >0$ for all $v \in (\K^n)^\times$, where
$$
(\K^n)^\times := \{ v \in \K^n \mid \exists \phi : \K^n \to \K \mbox{ linear map sth. } \phi(v) \in \K^\times \} .
$$
If $\K$ is square ordered, then the {\em standard form}
$q_0(x)=\sum_i x_i^2$ is positive definite, and every other form can be represented by a matrix in the usual way.
Symmetric positive definite matrices then correspond to positive definite forms, and they form the symmetric cone
of a poJa over $\K$. 
\end{example}

\begin{example}
The one-dimensional Jordan algebra $V=\K$ over $\K$ is a poJa if, and only if, $\K$ is a square-ordered inverse por. 
On the other hand, consider $\K=\Z$ and $V = \R$, with its usual  Jordan and cone structure. 
This example shows that  
invertible elements in the algebra need not be invertible in the base ring, so $\K$ need not be 
an inverse por. 
In general, if $E \subset V$ is an inclusion of unital Jordan algebras and
 $x \in E$, then $x$ may be invertible in 
$V$ and fail to be so in $E$.
\end{example}

\subsection{Formal reality}
Recall that a ($\R$-) Jordan algebra is called {\em formally real} if $\sum_{i=1}^n x_i^2 = 0$ implies that $x_i=0$ for all $i$.
Working with a general por $\K$, the following condition seems to be a reasonable analog of that condition:

\begin{definition}\label{def:formallyreal}
A poJa $V$ over a por $\K$ is called {\em formally real} the following is satisfied:
$\forall a \in V, \forall b \in V^\times$: $a^2 + b^2 > 0$
(in particular, $a^2+b^2$ is invertible).
\end{definition}

\begin{example}\label{ex:P*}
Assume that $\bA$ is a {\em $P^*$-algebra} (cf.\  \cite{Be17}, Appendix A),
 that is, a complex associative algebra with involution $*$ such that
$V = \{ a \in \bA \mid a^* = a \}$ is a poJa such that
$\forall a \in \bA, \forall b \in \bA^\times$: $a^* a + b^* b >0$.
Then $V$ satisfies the preceding condition.
For instance, any $C^*$-algebra is a $P^*$-algebra (but the converse is false).
\end{example}

\section{The Jordan geometry of a unital Jordan algebra}\label{sec:Jordangeo}

Let $V$ be a unital (quadratic) Jordan algebra over a commutative ring $\K$.
Then:

\subsection{Associated groups}
The {\em automorphism group} of $V$ is defined in the usual way, like for all algebras.
The {\em structure group} $\Str(V)$ of $V$ is defined by
\begin{equation}
\Str(V):= \{ g \in \Gl(V) \mid \,  \exists h \in \Gl(V):  \,
\forall x \in V: \, Q_{gx} = g \circ Q_x \circ h \} .
\end{equation}
It contains scalar multiples of the identity, and all invertible quadratic operators,
because of the fundamental formula (FF). Hence
$\Str(V)$ 
contains the {\em inner structure group}, which is the group
generated by all $Q_y$ with $y \in V^\times$:\footnote{ for general Jordan pairs,
the inner structure group is the group generated by the Bergmann operators $B(x,y)$ for quasi-invertible pairs
$(x,y)$; but for unital Jordan algebras, this comes down to the present definition because of the formula
$B(x,y) = Q_x Q_{x^{-1}-y}$, see \cite{Lo75}, I.2.12}
\begin{equation}
\Istr(V):= \langle Q_y \mid y \in V^\times \rangle \subset \Gl(V) .
\end{equation}
Finally, the {\em conformal}, or {\em Kantor-Koecher-Tits group} $\Co(V)$ associated to $V$ will
be defined below. The best and most rigourous definition of these groups is given in terms of
{\em Jordan pairs} (\cite{Lo75}), and constructing the {\em Kantor-Koecher-Tits (KKT) algebra} (thought of as the Lie algebra
of $\Co(V)$, see \cite{Lo95}) and \cite{BeNe04}. 

\subsection{Conformal completion and conformal group of $V$}
There exists a space $X=X(V)$, together with an imbedding $V \subset X(V)$, such that
\begin{enumerate}
\item
the Jordan inverse $x\mapsto x^{-1}$ extends to a map 
$j:X \to X$ which is of order two: $j \circ j = \id_X$,
\item
every translation $t_v:V \to V$, $x \mapsto x+v$ extends to a  bijective map
$X \to X$ (denoted by the same letter),
\item
every element $h \in \Str(V)$ extends to a bijective map $h:X \to X$,
\end{enumerate}

\nin
such that all relations satisfied by these maps on $V$ continue to hold for their extensions onto $X$. 
It follows that, for $v \in V$, we get another bijection
\begin{equation}
\tilde t_v := j \circ t_v \circ j : X \to X
\end{equation} 
such that $\tilde t_v \tilde t_w = \tilde t_{v+w}$.  In Loos' work \cite{Lo95},
the group
\begin{equation}\label{eqn:G0}
G = G_0(X) = \langle t_v, \tilde t_w \mid v,w \in V \rangle = \langle U^+,U^- \rangle
\end{equation}
 of bijections of $X$ generated by the (abelian) groups
$U^+ = \{ t_v \mid v \in V\}$ and $U^- = \{\tilde t_v \mid v \in V \}$, is called  the
{\em projective elementary group of $V$}, and
 is described by generators and relations.
It is realized as a subgroup of the automorphism group of the KKT-algebra.
The Jordan inverse $j$ also induces an automorphism of the KKT-algebra, but it need not be contained in
$G_0(X)$. However, it normalizes $G_0(X)$, and
 we may define slightly bigger groups:
\begin{equation}
G_1(X) = G_0(X) \cup  j G_0(X), \quad
\Co(V) := \langle G_0(X), \Str(V) \rangle,
\end{equation}
the latter corresponding to what is called the {\em conformal group} in \cite{Be00}.
All these groups act transitively on $X$. Let $o \in X$ be the origin of $0\in V$, and let
\begin{equation}
o' := \infty := j(o) \in X
\end{equation}
be its image under $j$ (``dual origin''). Then, as homogeneous space, 
\begin{equation}
X = G.o = G/P, \qquad X = G.o' = G.\infty = G/P',
\end{equation}
where $P'$ is a semidirect product of $t_V$ with the inner structure group (affine group), and
$P = jPj$ a semidirect product of $\tilde t_V$ with $\Istr(V)$.

\subsection{Transversality}
We say that a pair $(a,b)\in X \times X$ is {\em transversal}, and write $a \top b$, if there exists $g\in G$ with
$(a,b)=g.(o,\infty)=(g(o),g(\infty))$. Then, for every $a \in X$, the set
$a^\top = \{ b \in X \mid a \top b \}$, is an affine space (think of it as open dense in $X$, with complement
some set of ``points at infinity''). 
For any pair $(a,b) \in X^2$, let
\begin{equation}
U_{ab}:= a^\top \cap b^\top .
\end{equation}
For instance, if $(a,b) = (o,\infty)$, then $U_{ab} = V^\times$ is precisely the set of invertible elements in $V$,
and $U_{\infty,\infty} = V$. 
A triple $(a,b,c)$ with $a\top b, b\top c,a\top c$ is called a {\em transversal triple}.
For instance, $(o,e,\infty)$ is a transversal triple.
The existence of transversal triples distinguishes the Jordan geometries coming from unital Jordan algebras
among those coming from general Jordan pairs (and where
$X^+ = G/P$ and $X^-=G/P'$ are different spaces, thought of as ``dual to each other''). 

\subsection{Inversions}
For every pair $(a,b) \in X^2$, and every element $y \in U_{ab}$, there exists a unique element
$J^{ab}_y \in G_1(X)$ such that: 
\begin{equation}
J^{ab}_y(a)=b, \,
J^{ab}_y(b)=a, \,
J^{ab}_y(y)=y, \,
J^{ab}_y=J^{ba}_y, \,
(J^{ab}_y)^2 =\id_X, \,
J_y^{aa}= J^{yy}_a ,
\end{equation}
and, for any  $g \in \Co(V)$:
$J^{ga,gb}_{gy} = g \circ J^{ab}_y \circ g^{-1}$, and at the base points we have
\begin{equation}
J^{o\, o}_\infty (x) = -x = J^{\infty\infty}_o(x), \quad
J^{o \, \infty}_e(x) = j(x) = x^{-1} .
\end{equation}
In \cite{Be14}, (abstract) {\em Jordan geometries} have been characterized as spaces equipped with such families
of inversions, satisfying certain algebraic identities.
One of the consequences of these identities is that $U_{ab}$ is stable under the binary law
\begin{equation}
(y,x) \mapsto s_x(y):= J^{ab}_y(x),
\end{equation}
which turns $U_{ab}$ into a {\em symmetric space}, in the sense explained above.

\begin{example}
When the Jordan algebra comes from an associative algebra $\bA$, then $X$ is (a part of) the
\href{https://en.wikipedia.org/wiki/Projective_line_over_a_ring}{\em projective line $\bA\PP^1$ over the ring $\bA$}.
If $\bA$ is a (skew) field, then this is the usual one-point completion $\bA \cup \{ \infty\}$.
If $\bA = {\rm Fun}(M,\R)$ is the ring of all functions $f:M\to \R$, then 
$X(V) = \{ f :M\to \R \PP^1 \mid f \mbox{ function }  \}$
is simply the space of all functions from $M$ to $\R\PP^1$. 
More generally, the construction is compatible with general direct products. 
Note, however, that things are more involved for poJa's of continuous, or smooth, real functions: then $X$ is contained
in the space of continuous or smooth functions $M \to \R \PP^1$, but some analysis is necessary in order
 to say in which sense it may be dense there. 
\end{example}

\section{Cyclic orders}

\begin{definition}\label{def:pco}
 A {\em partial cyclic order (pco) } on a set $M$ is given by a ternary relation
$R \subset M^3$ such that:
\begin{enumerate}
\item
{\em
Cyclicity:} if $(a,b,c)\in R$, then $(b,c,a)\in R$,
\item
{\em Asymmetry}:  if $(a,b,c) \in R$, then $(c,b,a)\notin R$, 
\item
{\em  Transitivity}: if $(a,b,c)\in R$  and  $(a,c,d)\in R$, then $(a,b,d) \in R$.
\end{enumerate}
\nin
It is called {\em total} if, for every $(a,b,c) \in M^3$:

\ssk
either
$(a,b,c) \in R$, or $(a,c,b) \in R$, or ($a=b$ or $a=c$ or $b=c$).

\ssk
\nin
A {\em (strict) pco-morphism} between $(M,R)$ and $(M',R')$ is a map $f:M\to M'$ such that
$(a,b,c)\in R \Rightarrow (f(a),f(b),f(c)) \in R'$. 
\end{definition}

\nin
As is seen directly from (2) and (3), 
 for fixed $a$, we can define a (usual) partial order $<_a$ by:
  $b<_a  c$ iff $(a,b,c) \in R$. For the following lemma, recall Figure \ref{f:1}:

\begin{lemma}\label{la:quadruple}
Let $R \subset M^3$ a cyclic partial order and $(a,b,c,d) \in M^4$. Then the following are equivalent:
\begin{enumerate}
\item
$(a,b,c)\in R$ and $(a,c,d) \in R$
\item
$(a,b,d) \in R$ and $(b,c,d) \in R$
\item
any triple obtained by deleting one letter from the quadruple belongs to $R$
\end{enumerate}
\end{lemma}

\begin{proof}
By transitivity, (1) implies that
$(a,b,d) \in R$, and by cyclicity, that
($(c,a,b) \in R, (c,d,a) \in R$), whence by transitivity
$(c,d,b) \in R$, that is, $(b,c,d) \in R$, whence (2).
By the same kind of argument, (2) also implies (1).
Obviously, (1) and (2) together  are equivalent to (3). 
\end{proof}

\begin{definition}
Under the condition of the lemma, we say that $(a,b,c,d)$ forms a 
{\em cyclic quadruple}.
\end{definition}

\nin This defines a quaternary relation closely related to the {\em separation relation} used by Coxeter
\cite{Co47, Co49}: $(a,c)$ separates $(b,d)$ (but $(c,a)$ also separates $(b,d)$). 
-- 
The following proposition can be interpreted by saying that intervals are {\em convex}, in some general sense
(different from the one defined above for pom's): 

\begin{proposition}
Let $R$ be a pco on a set $M$, and $a,b \in M$. 
Assume $u,v \in ]a,b[$ are such that 
$u<_a v$. Then also  $u<_b v $, and  $]u,v[ \subset ]a,b[$. 
\end{proposition}

\begin{proof}
Our assumption implies that $(a,u,v,b)$ form a cyclic quadruple.
Let $x \in ]u,v[$, so $(u,x,v)\in R$.
Since $(a,u,v)\in R$, it follows that
the quadruple $(a,u,x,v)$ is cyclic. 
Likewise, $(u,x,v,b)$ is cyclic,
whence $(u,x,b) \in R$.
Since $(a,u,b) \in R$, it follows that
$(a,u,x,b)$ cyclic.
Thus $(a,x,b) \in R$, i.e., $x \in ]a,b[$.
\end{proof}


\begin{proposition}\label{prop:compr}
Let $g$ be an automorhism of $(M,R)$ such that $g(b)=b$ and $g(a) \in ]a,b[$.
Then $g(]a,b[)\subset ]a,b[$.
\end{proposition}

\begin{proof} 
Let $u:=g(a)$, so $(a,u,b)\in R$.
If $(a,x,b) \in R$, then
$(u,g(x),b) \in R$, and
 it follows that $(a,u,g(x),b)$ is a cyclic quadruple, whence
 $(a,g(x),b)\in R$.
\end{proof}

\section{The cyclic order defined by an ordered Jordan algebra}

\begin{theorem}\label{th:main}
Let $V$ be a poJa over an ordered ring $\K$, with symmetric cone $\Omega$.
Then its Jordan geometry  $X=X(V)$ carries a cyclic partial order,  given by
$$
R = \{ (a,x,b) \in X^3 \mid \, \exists g \in G_0: \,\,
g(a)=o, \, g(b) = \infty, \,  g(x) \in \Omega \} 
$$
with $G_0 = G_0(X)$ given by (\ref{eqn:G0}).
In other words, $R$ is defined by the intervals 
$$
]a,b[ = g(\Omega) \qquad  \mbox{ if } \qquad    g.(o,\infty)=(a,b) , \, g \in G_0 \,  .
$$
This cyclic partial order is uniquely characterized by:
\begin{enumerate}
\item
it is $G_0$-invariant, and
\item
it coincides with the given poJa on $V$:
$(a,x,\infty) \in R$ iff $a<x$ iff $x-a \in \Omega$.
\end{enumerate}
The cyclic order 
 is reversed by all inversions
$s=J^{uv}_w$, i.e., 
$s(]a,b[) = ]s(b),s(a)[$. 
When $a \top b$, then the interval $]a,b[$ is non-empty, and
it is a symmetric subspace of $U_{ab}$, 
 isomorphic to $\Omega$ as symmetric space. 
\end{theorem}

\begin{proof}
We  fix the transversal pair $(o,o')=(o,\infty)$ as origin in $X^2$, and use notation from Section \ref{sec:Jordangeo}.
Defining $R$ as in the theorem, let us first show that 
$]o,\infty[=\Omega$. Since the stabilizer of $(o,\infty)$ in $G(X)$ is the inner structure group, this follows from

\begin{lemma}
The inner structure group $\Istr(V)$ of an ordered Jordan algebra preserves its symmetric cone $\Omega$.
\end{lemma}

\begin{proof}
This follows from the fact that $\Omega$ is invariant under all invertible quadratic operators $Q_x$, and that $\Istr(V)$ is generated by
such operators.  
\end{proof}

By definition, it is obvious that the set $R \subset X^3$ is invariant under $G_0(X)$, and
since all triples of the form $(o,x,\infty)$ ($x\in V^\times$)
are transversal triples, it follows that $R$ is contained in the set of 
transversal triples. Fixing $b=\infty$, 
if $(a,x,b) \in R$, since $(a,x,b)$ is an transversal triple, 
we have $a,x \in \infty^\top = V$.
Since the stabilizer of $\infty$ acts affinely on $V$ (semidirect product of translations and $\Istr(V)$), it now follows that
$(a,x,\infty) \in R$ if, and only if, $x-a \in \Omega$, and thus the partial order on $V$ coincides with $R$
in the sense that, for all $a,x \in V$,
\begin{equation}
x < a \mbox{ iff } (a,x,\infty) \in R .
\end{equation}
Let us now prove that $R$ satisfies the defining properties of a partial cyclic order.
\begin{enumerate}
\item
cyclicity: by \cite{Be14}, Theorem 6.1,
every cyclic permutation of a transversal triple is induced by an element of $G_0$, hence
cyclicity follows from invariance of $R$ under the action of $G_0$
(if $(a,x,b)=(0,1,\infty)$, then the two non-trivial cyclic permuations can be defined by
$g(x)=1-x^{-1}$ and $h(x)=(1-x)^{-1}$),
\item
asymmetry: let $(a,x,b) \in R$; by invariance under $G_0$, we may assume $b=\infty$, so 
$a<x$ in $V$. By the property of partial order on $V$,  we have (not$(x<a)$), that is
$(x,a,b)\notin R$,
\item
transitivity: by cyclicity and invariance under $G_0$, we may assume $a=\infty$; so
$(a,x,y) \in R$ iff $(x,y,a) \in R$ iff 
 $x<y$ in $V$. Now transitivity of $R$ corresponds to transitivity of $<$ on $V$, which holds by assumption.  
\end{enumerate}

\nin
Assume $(a,b)$ is a transversal pair. 
Let us show that $]a,b[$ is a symmetric subspace of $U_{ab}$, isomorphic to $\Omega$ (in particular, non empty).
Since the symmetric space structure of $U_{ab}$ is invariant under the stabilizer of $(a,b)$, 
by transitivity of $G_0$ on the set of transversal pairs, 
we may again
assume that $(a,b)=(o,\infty)$; then $U_{ab}=V^\times$ with the symmetric space structue
$s_x(y)=Q(x)y^{-1}$, and
 as noted in section \ref{ssec:Ja} , $\Omega$ is a symmetric subspace of $V^\times$. 

\ssk
Finally, we show that inversions reverse $R$.
Since the composition of any two inversions belongs to $G_0$, it suffices to show that one particular inversion
reverses $R$.
This is most easily seen for the inversion $J_0^{\infty,\infty}(x)=-x$, which fixes $\infty$ and obviously
reverses the order of $V$.
\end{proof}

\begin{definition}
The subgroup of $\Co(V)$  preserving the partial cyclic order $R$ is called the
{\em causal group of $X$} and denoted by
$$
\Cau(V,\Omega):=
{\rm Cau}(X,R) = \{ g \in \Co(V)  \mid g.R = R \} = \Co(V) \cap \Aut(R),
$$
and we let also
$$
G(\Omega):= \{ g \in \Str(V) \mid \, g(\Omega) = \Omega \} = \Str(V) \cap \Aut(R) .
$$
\end{definition}

\nin
By the theorem, $G_0(X) \subset {\rm Cau}(X,R)$. Under certain conditions, a converse holds:

\begin{theorem}
Assume $V$ is a  Euclidean Jordan algebra 
containing no direct factor isomorphic to $\R$. Then any 
automorphism of $R$ which is of class ${\mathcal C}^4$
(four times continuously differentiable), is given by an element of $\Cau(X,R)$:
$$
\Aut_{{\mathcal C}^4}(X,R)  = \Cau(X,R) .
$$
\end{theorem}

\begin{proof}
If $g:X\to X$ is of class ${\mathcal C}^4$ and preserves $R$, its differential $T_x g$ at $x$ sends the cone at $x$ to the cone at
$g(x)$, hence $g$ is a {\em causal diffeomorphism}.
Now the claim follows from the generalized Liouville Theorem \cite{Be96}, or \cite{Be00}, Th.\ IX.2.4.
\end{proof}
 
\nin 
Of course, in the general setting of poJa's, the theorem will not always carry over. 
Also, in general we will be very far from a classifcation of ${\Cau}(X,R)$-orbits in the space of pairwise transversal
triples: for a simple Euclidean Jordan algebra of rank $r$, there are $r+1$ orbits, characterized by the signature
(one of these orbits is $R$, another $R^{op}$)
but for base fields such as $\Q$ the classification is much more complicated. 
 --
Just as a projective conic can be described by different kinds of affine image (ellipse, hyperbola, parabola), 
so can the intervals $]a,b[$ (cf.\ \cite{Be00}, Theorem XI.3.3; recall also Figure \ref{f:2} from the Introduction):

\begin{theorem}\label{th:affineimage}
With notation as in the preceding theorem, let $(a,b)$ be a transversal pair. Then the interval $]a,b[$ has affine
image $]a,b[\cap V$ as follows:
\begin{enumerate}
\item
(parabolic image) if $a \in V$ and $b=\infty$, then $]a,b[ = a +\Omega$, 

if $b\in V$ and $a =\infty$, then $]a,b[ = b - \Omega$,
\item
(elliptic image)
if $a,b \in V$ and $a<b$, then 
$]a,b[ = (a + \Omega) \cap (b-\Omega)$; this is a  convex subset of $V$,
\item
(hyperbolic image)
If $a,b \in V$ and  $b<a$, then $ (a+\Omega) \cup (b-\Omega)$  is contained in
$]a,b[ \cap V $  (but equality does not hold in general).
If $a,b \in V$ and (not $a<b$), then $]a,b[$ is in general not contained in $V$, and
 $]a,b[\cap V$ is in general not a convex subset of $V$. 
 \end{enumerate}
\end{theorem}

\begin{proof}
(1): by Theorem \ref{th:main}, $]0,\infty[=\Omega$, and applying
$g = \tau_a$ (translation), the first part of the claim follows. 
For the second part, by translation, we may assume $b=0$. Then $g(x)=-x^{-1}$ is a composition of two inversions, hence belongs
to $G_0$, it sends $]0,\infty[$ to $]\infty,0[$, and also $\Omega$ to  $ - \Omega$, whence $]\infty,0[=-\Omega$. 

To prove (2), assume
$a,x,b \in V$.
Then by Lemma \ref{la:quadruple} the following are equivalent, 

\ssk
($a<b$ and $x \in ]a,b[$),

($(a,b,\infty) \in R$ and $(a,x,b) \in R$),

$(a,x,b,\infty)$ is a cyclic quadruple,

($(a,x,\infty) \in R$ and $(\infty,x,b) \in R$),

($x \in a +\Omega$ and $x \in b - \Omega$), 

\ssk

\nin giving the claim.  As to case (3), in general not much can be said.
Let's  just consider the special case $b<a$, i.e., $(b,a,\infty) \in R$. 
Assume first $x \in a+\Omega$, so $(a,x,\infty) \in R$. Then $(b,a,x,\infty)$ is a cyclic quadruple, hence
$(a,x,b) \in R$, whence $x \in ]a,b[$.
Likewise, when $x \in b-\Omega$, it follows that $x \in ]a,b[$. 
But one cannot reverse these arguments since Lemma \ref{la:quadruple} does not apply in the situation
$(b,a,\infty) \in R, (a,x,b) \in R$. 
\end{proof}

Working in finite dimension over $\R$, one may decompose the set $U_{ab} \cap V$ into topological connected 
components, and $]a,b[ \cap V$ will be the union of certain of them. However, the algebraic equations  of 
$V \setminus U_{ab}$ are polynomial of high degree, and the image will in general be very complicated. 
The following example gives a slight impression, in a  simple situation:

\begin{example}[pco on the $n$-torus]\label{ex:torus}
Let $V =\K =  \R$. Then $X = \R \PP¨^1$,  which will be identified with the
circle. The cyclic order can be represented by Figure \ref{f:1}. The Jordan chart $V$ is given by stereographic projection,
but it's easier just to identify $S^1$ with $\R / 2 \Z$, and $V$ with $]-1,1[$. 
An elliptic image of an  interval is $]a,b[$ with $-1<a<b<1$, a parabolic one $]a,1[$ or $]-1,b[$, and a hyperbolic one
$]a,1[\cup ]-1,b[ = ]a,\infty[\cup ]\infty,b[$ with $b<a$. 

\ssk
Now let $V = \R^n$ with cone $\Omega =\{ x \in \R^n \mid \forall i : x_i >0\}$. This is a direct product of copies of $\R$, and
the {\em $n$-torus} $(S^1)^n = \R^n / 2 \Z^n$ is its Jordan geometry.
The pco on the $n$-torus is the $n$-fold direct product of the cyclic order described above.
Affine images can be drawn by
considering $a,b$ in the open $n$-cube $]-1,1[^n$. 
(Again, the Jordan chart arises by stereographic projection in each component, stretching the $n$-cube to all of $\R^n$.)
When $a<b$, then we are in the elliptic case (cf.\ Figure \ref{f:2}).
Consider the case $b<a$: then the affine image of $]a,b[$ is given by
$$
]a,b[\cap V = 
\prod_{i=1}^n ]a_i,b_i[ = \prod_{i=1}^n \bigl( ]a_i,\infty_i[ \cup ]\infty_i,b_i[ \bigr) =
\prod_{i=1}^n \bigl( ]a_i,1[ \cup ]-1,b_i[ \bigr) ,
$$
leading to a union of $2^n$ affine cubes. In the general case, say with $k$ components $b_i <a_i$, the same argument leads
to a union of $2^k$ affine cubes. 
Figure \ref{f:3} shows, for $n=2$, first the affine image of the cone $\Omega = ]0,1[^n$ (parabolic), second, the hyperbolic case
$b<a$, and third, the hyperbolic case $a_1<b_1, b_2 < a_2$.
\begin{figure}[h]
\caption{Three affine  images of an interval in a $2$-torus}\label{f:3}
\newrgbcolor{ffxfqq}{1. 0.4980392156862745 0.}
\psset{xunit=1.0cm,yunit=1.0cm,algebraic=true,dimen=middle,dotstyle=o,dotsize=5pt 0,linewidth=0.8pt,arrowsize=3pt 2,arrowinset=0.25}
\begin{pspicture*}(-5.944128022950112,-0.4073633896218)(7.688238282393222,2.497743342935735)
\pspolygon[linecolor=ffxfqq,fillcolor=ffxfqq,fillstyle=solid,opacity=0.1](1.3864011433972339,1.2976147026387934)(1.3864011433972339,2.)(2.,2.)(2.,1.2976147026387934)
\pspolygon[linecolor=ffxfqq,fillcolor=ffxfqq,fillstyle=solid,opacity=0.1](1.3864011433972339,0.)(1.3864011433972339,0.5469534137346596)(2.,0.5469534137346596)(2.,0.)
\pspolygon[linecolor=ffxfqq,fillcolor=ffxfqq,fillstyle=solid,opacity=0.1](0.,2.)(0.,1.2976147026387934)(0.8323416206346599,1.2976147026387934)(0.8323416206346599,2.)
\pspolygon[linecolor=ffxfqq,fillcolor=ffxfqq,fillstyle=solid,opacity=0.1](0.,0.5469534137346596)(0.,0.)(0.8323416206346599,0.)(0.8323416206346599,0.5469534137346596)
\pspolygon[linecolor=ffxfqq,fillcolor=ffxfqq,fillstyle=solid,opacity=0.1](5.7473857741736225,1.243996039145641)(6.319318184767247,1.243996039145641)(6.319318184767247,2.)(5.7473857741736225,2.)
\pspolygon[linecolor=ffxfqq,fillcolor=ffxfqq,fillstyle=solid,opacity=0.1](5.7473857741736225,0.36822453542415146)(6.319318184767247,0.36822453542415146)(6.319318184767247,0.)(5.7473857741736225,0.)
\pspolygon[linecolor=ffxfqq,fillcolor=ffxfqq,fillstyle=solid,opacity=0.1](-4.,1.)(-4.,2.)(-3.,2.)(-3.,1.)
\psline[linecolor=ffxfqq](1.3864011433972339,1.2976147026387934)(1.3864011433972339,2.)
\psline[linecolor=ffxfqq](1.3864011433972339,2.)(2.,2.)
\psline[linecolor=ffxfqq](2.,1.2976147026387934)(1.3864011433972339,1.2976147026387934)
\psline[linecolor=ffxfqq](1.3864011433972339,0.)(1.3864011433972339,0.5469534137346596)
\psline[linecolor=ffxfqq](1.3864011433972339,0.5469534137346596)(2.,0.5469534137346596)
\psline[linecolor=ffxfqq](2.,0.5469534137346596)(2.,0.)
\psline[linecolor=ffxfqq](0.,1.2976147026387934)(0.8323416206346599,1.2976147026387934)
\psline[linecolor=ffxfqq](0.8323416206346599,1.2976147026387934)(0.8323416206346599,2.)
\psline[linecolor=ffxfqq](0.8323416206346599,2.)(0.,2.)
\psline[linecolor=ffxfqq](0.,0.5469534137346596)(0.,0.)
\psline[linecolor=ffxfqq](0.8323416206346599,0.)(0.8323416206346599,0.5469534137346596)
\psline[linecolor=ffxfqq](0.8323416206346599,0.5469534137346596)(0.,0.5469534137346596)
\psline[linecolor=ffxfqq](5.7473857741736225,1.243996039145641)(6.319318184767247,1.243996039145641)
\psline[linecolor=ffxfqq](6.319318184767247,1.243996039145641)(6.319318184767247,2.)
\psline[linecolor=ffxfqq](6.319318184767247,2.)(5.7473857741736225,2.)
\psline[linecolor=ffxfqq](5.7473857741736225,2.)(5.7473857741736225,1.243996039145641)
\psline[linecolor=ffxfqq](5.7473857741736225,0.36822453542415146)(6.319318184767247,0.36822453542415146)
\psline[linecolor=ffxfqq](6.319318184767247,0.36822453542415146)(6.319318184767247,0.)
\psline[linecolor=ffxfqq](5.7473857741736225,0.)(5.7473857741736225,0.36822453542415146)
\psline(0.,0.)(2.,0.)
\psline(2.,0.)(2.,2.)
\psline(2.,2.)(0.,2.)
\psline(0.,0.)(0.,2.)
\psline(5.,0.)(7.,0.)
\psline(5.,0.)(5.,2.)
\psline(5.,2.)(7.,2.)
\psline(7.,0.)(7.,2.)
\psline(-5.,0.)(-3.,0.)
\psline(-3.,0.)(-3.,2.)
\psline(-3.,2.)(-5.,2.)
\psline(-5.,2.)(-5.,0.)
\psline(-4.,1.)(-4.,2.)
\psline(-4.,2.)(-3.,2.)
\psline(-3.,2.)(-3.,1.)
\psline(-4.,1.)(-3.,1.)
\psline[linecolor=ffxfqq](-4.,1.)(-4.,2.)
\psline[linecolor=ffxfqq](-4.,2.)(-3.,2.)
\psline[linecolor=ffxfqq](-3.,2.)(-3.,1.)
\psline[linecolor=ffxfqq](-3.,1.)(-4.,1.)
\begin{scriptsize}
\psdots[dotstyle=*](0.8323416206346599,0.5469534137346596)
\rput[bl](0.8680873962967599,0.7793009555383194){{$b$}}
\psdots[dotstyle=*](1.3864011433972339,1.2976147026387934)
\rput[bl](1.4936384703835368,1.01164849734198){{$a$}}
\psdots[dotstyle=*](5.7473857741736225,1.243996039145641)
\rput[bl](5.568656895863112,0.8507925068625226){{$a$}}
\psdots[dotstyle=*](6.319318184767247,0.36822453542415146)
\rput[bl](6.462301287415651,0.4933347502415063){{$b$}}
\end{scriptsize}
\end{pspicture*}
\end{figure}
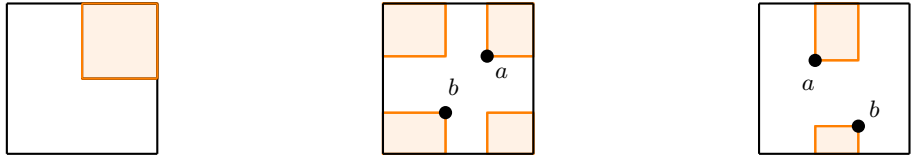
\end{example}

\begin{example}[Lorentz cones]
In case of the Euclidean Jordan algebra $V = \R^{n-1,1}$, with symmetric cone the Lorentz cone, the Jordan geometry is
the projective quadric given by a form of signature $(n,2)$ in
in $\R^{n+2}$. Since the cone is given by quadratic equations $q(x)>0,x_0>0$, one can give explicit descriptions of
all three types of images, by looking at signs of $q(x-a)$ and $q(x-b)$. The intersection of the null-cones $q(x-a)q(x-b)=0$
is given by ellipsoids and hyberboloids in a hyperplane of $V$, which leads to pictures having a basic structure like those
of Figure \ref{f:3}, with circular cones replacing the right angles at $a$ and $b$.
\end{example}

\begin{definition}\label{def:intervaltopology}
With notation and assumptions as in Theorem \ref{th:main}, the {\em interval topology}, or {\em order topology},
on $X$ is the topology generated by all intervals $]a,b[$, with $(a,b) \in X^2$ such that $a \top b$. 
\end{definition}

It is clear from the definition that $\Cau(X,R)$ acts by homeomorphisms on $X$. 

\begin{theorem}\label{th:topology}
When $V$  is a Euclidian Jordan algebra, or a Jordan-Banach algebra, then the interval topology coincides with
the usual topology on $X$, and by restriction, it also coincides with the usual topology on $V$.
\end{theorem}

\begin{proof}
Recall from \cite{BeNe05} that topologies on $X$ correspond to topologies on $V$ having certain invariance properties,
and then all chart domains $U_a$ are open in $X$. 
Thus it suffices to compare the usual topology of $V$ with
the interval topology generated by the affine images of intervals on $V$.
Now,  elliptic images of intervals are precisely  the unit balls of the spectral norm, and since the spectral norm
is equivalent to the usual norm, it follows that  the usual topology is contained in the interval topology.
To prove the converse, it suffices to note that all affine images (the hyperbolic ones included) are open in the usual topology.
But this is clear since $\Cau(X,R)$ acts by homeomorphisms, and we have already seen that elliptic images are open
in $V$, hence in $X$. 
\end{proof}

In the general case, the interval topology need not be Hausdorff:

\begin{example}\label{ex:TX}
Let $V$ be Euclidean (or JB), and consider its ``tangent algebra''
$TV:= V \oplus \eps V$ with $\eps^2 = 0$ (scalar extension by dual numbers $\R[\eps]$).
Then $TV$ is a poJa with cone $T\Omega = \Omega + \eps V$ (rather a wedge), and $X(TV)=TX$ is the {\em  tangent
bundle} of $X$.  Intervals are all of the form
$]A,B[ = T (]a,b[)$ with $a=\pi(A), b=\pi(B)$ the images of $A,B$ under the canonical projection $TX\to X$. 
Thus points $u,v \in T_x X$ living in the same tangent space of $X$ cannot be seperated by the interval topology.
\end{example} 

The interval topology should be a valuable tool when attacking the following problems.

\section{Open problems and further topics}\label{sec:further}

\subsection{Tube domains, bounded symmetric domains}
Basic definitions related to these topics carry over to our general setting: 
 any por and any poJa admit  a ``complexification''
($V[i]$: scalar extension by $\K[i]$, cf. example \ref{exa:C}), and all basic definitions and constructions in principle 
continue to make sense: we may define the {\em  tube-domain}  $T_\Omega = V +i\Omega \subset V[i]$, on which
the translation group and the group $G(\Omega)$ act. 
Now, it is natural to ask: {\em is $T_\Omega$ a symmetric domain, in the sense that about any point there is a
``holomorphic'' symmetry, turning $T_\Omega$ into a symmetric space?}
As in the classical case, the question reduces to: {\em does inversion at $i$, given by $z \mapsto - z^{-1}$,
define a bijection  $T_\Omega \to T_\Omega$? is every element of $T_\Omega$ invertible in $V[i]$?}
The proofs given in the classical cases (\cite{FK94, Up}) use analytic arguments, and hence do not carry
over to general poJa's. 
However, the answer is certainly positive for the $\Q$-poJa $\Sym(n,\Q)$
(since $a+ib$ invertible as real matrix implies it is also invertible as rational matrix), and it is also positive for the case of
$P^*$-algebras (example \ref{ex:P*}; cf.\ \cite{Be17}). 
So the problem is: {\em give a necessary and sufficient condition, in terms of Jordan geometry and order
theory, for $T_\Omega$ to be a symmetric domain.}

\ssk
Via the {\em Cayley transform}, the preceding problem translates to the question whether the `Hermitian'', or
``twisted complexification'' (cf.\  \cite{Be00, BeNe04, Be14}), of a bounded interval remains bounded.

\ssk
Results of Koufany on the {\em compression semigroup of a symmetric cone}, $\Gamma := \{ g \in \Cau(X,R) \mid g(]a,b[)\subset ]a,b[ \}$
(cf.\ \cite{K}, Th\'eor\`eme 4.9), carry over (cf.\ also Prop.\  \ref{prop:compr}); but again it is not clear what can be said about compression semigroups of tube domains or
of Hermitifications of intervals. 
As shown in \cite{FG96}, convexity is important in these questions --
for non-convex cones, the corresponding statements are false: their twisted
complexifications always have ``points at infinity''.

\subsection{Compact dual, symmetric $R$-spaces, Borel imbedding}
By the general theory from \cite{Be00, Be14}, 
the $c$-dual(=``Cartan dual'') symmetric space of any $U_{ab}$ can be realized inside the same
Jordan geometry $X(V)$. 
Now, the symmetric space  $U_{ab} = \Omega$, a convex cone, should be thought of as being of ``non-compact type''.
Thus its $c$-dual, written $ \Omega^c$,  should be thought of as being ``compact-like''. When 
do we have equality $X(V) = \Omega^c$? (For Euclidean Jordan algebras this holds since both spaces are compact: we realize
 $X$ as a ``symmetric $R$-space''.)
 When do we have, at least, inclusion $\Omega \subset \Omega^c$ ? (``Borel imbedding'')

\subsection{The boundary of intervals, and its structure}
First of all, find the correct definition of the ``maximal cone'', that is, extend the partial order on $V$, and the cyclic 
order on $X$, also to non-transveral pairs or triples.
A little care is needed here: the first guess might be take the closure $\overline \Omega$
 of $\Omega$ with respect to the interval
topology, but that cone might be too big. One should at least exclude the closure of $\{ 0 \}$
(because of non-Hausdorff cases such as the one from Example \ref{ex:TX}), so possibly work with the
cone
$$
C := \overline \Omega \setminus \overline{\{ 0 \}}. 
$$
This should define a cone such that
$C \cap C =\emptyset$, $Q_y(C) = C$ for all invertible $y$, $C+C \subset C$, which are the basic properties needed
to globalize it along the lines of proof of Theorem \ref{th:main}. 
The ternary relation on $X$ thus obtained will be asymmetric and transitive, but
it is not clear at all whether it will be cyclic.
Once the correct definition of the ``boundary'' $\partial \Omega := C \setminus \Omega$ being
clear, one will wish  to analyze it according to the model of the classical theory (see \cite{FK94, Up}).

\subsection{Relation with generalized cross-ratios, and Maslov index}
The classical cross-ratio $[a,b;c,d]$ of a quadruple of points on the real projective line is positive if, and only if,
$a$ and $b$ belong both to $]c,d[$ or both to $]d,c[$, and negative iff $(c,d)$ ``separates'' $(a,b)$
(one of them is in $]a,b[$ and the other in $]b,a[$).
There are generalizations of the cross-ratio for Jordan geometries (due to Kantor, and others), and there should be analogs
of this property. Similarly, it should be possible to characterize the cyclic order by a generalized Maslov index
(see \cite{NO}), whenever this index is defined.


\subsection{Structure theory; traces and states}
poJa's and their morphisms form a very rich category: it contains the whole variety of members, from nilpotent 
to simple ones. Thus one may try to tidy up by developing a structure theory in the usual algebraic sense.
For instance, what is the relation between
 {\em semisimple} poJa's and the {\em formally real} ones (in the sense of Def.\ \ref{def:formallyreal})? 
Note that 
 a general poJa may not admit traces, nor have non-trivial ``states''. 
It is part of  structure theory to give the correct definitions here, and to clarify the role of 
analogs of Euclidean Jordan algebras, Jordan-Hilbert, and Jordan-Banach algebras in the general setting. 


\subsection{About ``dual cones'' and ``symmetric cones''}\label{ssec:dualcone}
As said in the introduction, ``duality'' in the context of general Jordan structures means something different from
the functional analytic approach via duality of topological vector spaces, and likewise, ``dual cones'' living in the
topological dual of $V$ play no r\^ole at this stage of the theory. Nevertheless, {\em duality} is an important aspect of
Jordan theory, which (in my opinion) is best understood if one widens the scope from Jordan algebras to general
{\em Jordan pairs}. 
Our general ``symmetric cones'', as defined here, seem very well to be {\em self-dual} in some more
abstract geometric sense. 
I don't know how to formulate this property in a purely order-theoretic and ``cone-theoretic'' way, but essentially
it should express that a cone $\Omega$ should better be seen as an ``interval'', as the elliptic image on the right of
Figure \ref{f:2}, and not as the ``parabolic'' image on the left. The elliptic image really is complete, whereas the
parabolic image sort of hides half of the boundary. This suggests to call a cone, or an interval,
{\em symmetric} if it
 is symmetric about any of its points: there is an order-reversing symmetry
of $\Omega$ fixing that point.
 Such ``inversions'' extend to a common completion $X$, and then the axioms of Jordan geometries from
 \cite{Be14} are quite natural geometric compatibility conditions.
Thus ``symmetric intervals'' would be, indeed, closely tied to  Jordan geometries.

\end{document}